\documentclass[12pt,twoside,reqno,final]{amsart}
\usepackage{amsmath}
\usepackage{amssymb}
\usepackage{graphicx}
\usepackage{hyperref, verbatim}

\input epsf

\def\sideremark#1{\ifvmode\leavevmode\fi\vadjust{\vbox to0pt{\vss 
      \hbox to 0pt{\hskip\hsize\hskip1em           
 \vbox{\hsize2cm\tiny\raggedright\pretolerance10000
 \noindent #1\hfill}\hss}\vbox to8pt{\vfil}\vss}}}%

\newtheorem{theorem}{Theorem}

\newtheorem{proposition}{Proposition}
\newtheorem{lemma}{Lemma}

\theoremstyle{definition}

\newtheorem{example}{Example}
\theoremstyle{remark}

\renewcommand{\epsilon}{\varepsilon}

\newcommand{\R}{\mathbb{R}}
\newcommand{\s}{\mathbb{S}}
\newcommand{\C}{\mathbb{C}}

\newcommand{\OO}{\mathcal{O}}
\newcommand{\AC}{\mathbb{A^\C}}
\newcommand{\BC}{\mathbb{B^\C}}
\newcommand{\A}{\mathbb{A}}
\newcommand{\B}{\mathbb{B}}

\def\Ind{\operatorname{Ind}}
\def\dim{\operatorname{dim}}
\def\sgn{\operatorname{sgn}}
\def\Ann{\operatorname{Ann}}
\def\Hess{\operatorname{Hess}}
\usepackage{showkeys}

\title[Index of Real Vector Fields]{Index of Singularities \\ of Real Vector Fields \\ on Singular Hypersurfaces}
\author{Pavao Marde\v si\'c}
\address{Universit\'e de Bourgogne\\
Institut de Math\'ematiques de Bourgogne-- UMR 5584 du CNRS\\
UFR Sciences et Techniques\\
9, Avenue Alain Savary\\
BP 47870\\
21078 DIJON\\
France}

\begin{document}

\begin{flushright}
\emph{...with affection and respect,\\ for all the pleasure of working with Xavier }

\vskip1cm
\end{flushright}

\begin{abstract} 
G\'omez-Mont, Seade and Verjovsky introduced an index, now called GSV-index, generalizing the Poincar\'e-Hopf index to complex vector fields tangent to singular hypersurfaces. The GSV-index extends to the real case. 

This is a survey paper on the joint research with G\'omez-Mont and Giraldo about calculating the GSV-index $\Ind_{V_\pm,0}(X)$ of a real vector field $X$ tangent to a singular hypersurface $V=f^{-1}(0)$. 
The index $\Ind_{V_{\pm,0}}(X)$ is calculated as a combination of several terms. Each term is given as a signature of some bilinear form on a local algebra associated to $f$ and $X$. 
Main ingredients in the proof are G\'omez-Mont's formula for calculating the GSV-index on \emph{ singular complex} hypersurfaces and the formula of Eisenbud, Levine and Khimshiashvili for calculating the Poincar\'e-Hopf index of a singularity of a \emph{real} vector field in $\R^{n+1}$.
 
\end{abstract}
\maketitle
\section{Introduction}
This paper is a survey of the joint work with Xavier G\'omez-Mont and Luis Giraldo spread over some 15 years. We give a formula for calculating the index of singularities of real vector fields on singular hypersurfaces. Some partial results are published in \cite{GGM1}, \cite{GM1}, \cite{GM2},  \cite{GGM2}. 

In \cite{GSV}, G\'omez-Mont, Seade and Verjovsky studied vector fields tangent  to a complex hypersurface with isolated singularity. They introduced a notion of index, now called GSV-index
 at a common singularity of 
the vector field and the hypersurface (see also \cite{BG}). 
It is a kind of relative version of the Poincar\'e-Hopf index at a singularity. 
A natural question is how can one calculate this 
index. 
Complex case was studied first. It was solved by G\'omez-Mont in his seminal paper \cite{G}. G\'omez-Mont's formula expresses the GSV index via dimensions of certain local algebras. The GSV index can be generalized to the real case. More precisely, depending on the side of the singular hypersurface, there are two GSV indices. 
Real case, is more difficult than the complex case since in the real case a simple singularity can carry the index $+1$ or $-1$, whereas in the complex case all simple singularities count as $+1$. 

In the absolute  real case Eisenbud, Levine and Khimshiashvili expressed the Poincar\'e-Hopf index of a vector field in terms of the signature of a bilinear form. 

Our result in the relative real case expresses the GSV-index of a real vector field on a singular variety as a sum of certain terms. Each term is a signature of a  non-degenerate bilinear form on some local algebra. 

Our proof has two essential ingredients:  on one hand Gomez-M\'ont's result in the singular complex case and on the other hand the Eisenbud, Levine, Khimshiashvili's result in the real absolute case.  

\subsection{Real absolute case}
Let us recall first the definition of the Poincar\'e-Hopf index of a singularity of a real vector field in $\R^{n+1}$. 
Let 
\begin{equation}\label{X}
X=\sum_{i=0}^n X^i\frac{\partial}{\partial x_i}
\end{equation}
 be a smooth vector field in $\R^{n+1}$ having an isolated singularity at the origin $X_0=0$. One can identify the vector field $X$ with a mapping 
$X:(\R^{n+1},0)\to(\R^{n+1},0)$. Taking a small sphere $S^{n}$ around the origin, the vector field $X$ induces a map $N=\frac{X}{||X||}:S^n\to\s^n$, where $\s^n$ is the unitary sphere in $\R^{n+1}$.
The Poincar\'e-Hopf index $\Ind_0(X)$ of the vector field $X$ at the origin is defined as the degree of $N$.
That is, $\Ind(X,0)$ is the number of pre-images of generic points taken with orientation.
 
 \begin{example} Let $X$ be the vector field  $X(x,y)=x\frac{\partial}{\partial x}+y\frac{\partial}{\partial y}$ in $\R^2$ having a node at the origin and let $Y$ be the vector field $Y(x,y)=x\frac{\partial}{\partial x}-y\frac{\partial}{\partial x}$ having a saddle at the origin.
 
 Then $\Ind_0(X)=1$ and $\Ind_0(Y)=-1$.
 
 \end{example}

\subsection{Complex absolute case}
Consider the complex $n$-dimensional space $\C^n$, with complex coordinates $x_1,\ldots,x_n$ and a complex vector field $X$ of the form $X=\sum_{i=0}^n X^i\frac{\partial}{\partial x_i}$.
We can identify $\C^n$ with $\R^{2n}$. With this identification a holomorphic vector field on $\C^n$ becomes a smooth real vector field on $\R^{2n}$
and one can apply the previous definition of the Poincar\'e-Hopf index $\Ind_0(X)$ to a singularity of a holomorphic vector field. 
Note that not every smooth real vector field on $\R^{2n}$ comes from a holomorphic vector field on $\C^{n}$. By holomorphy, a holomorphic vector field seen as a map preserves orientation. Hence the index of a singularity of a holomorphic vector field is necessarily positive. 

\begin{example} Let $n=1$ and let $X=x\frac{\partial}{\partial x}$ and $Y=x^2\frac{\partial}{\partial x}$ be vector fields in $\C$.
Then $\Ind_0(X)=1$ and $\Ind_0(Y)=2$. 
\end{example}

In the complex case, the Poincar\'e-Hopf index is simply the multiplicity. One counts how many points are hidden at the singularity at the origin. 

\section{Definition of the GSV-index in the complex and real case}

\subsection{Smooth points}
Let now $f:(\R^{n+1},p)\to(\R,0)$ be a germ of an analytic function. Then $V=f^{-1}(0)$ is a germ of a hypersurface at $p$.  We say that a vector field defined in a neighborhood of $p\in V$ is a vector field tangent to $V$, if there exists an analytic function $h$ such that

\begin{equation}\label{tangent}
X(f)=fh.
\end{equation}
The function $h$ is sometimes called the \emph{cofactor} of $X$. Assume first that $p\in V$ is a regular point of $f$. Then the variety $V$ is smooth in  a neighborhood of $p$. Let $x=(x_1,\ldots,x_n)$ be a chart of $V$ in a neighborhood of $p$. We assume moreover that the orientation of $\nabla f,\frac{\partial}{\partial x_1},\ldots,\frac{\partial}{\partial x_n}$ is positive. The chart $x=(x_1,\ldots,x_n)$ transports the vector field $X$ to $\R^n$. One then applies the usual definition of the Poincar\'e-Hopf index.
Thus we define the \emph{relative Poincar\'e-Hopf index} $\Ind_{V,p}(X)$ of a vector field tangent to a hypersurface, relative to the surface. It is easy to verify that the definition is independent of the choices. 

If $f:(\C^{n+1},p)\to(\C,0)$ is a germ of holomorphic function instead, $p\in\C^{n+1}$ is a regular point of $f$, $V=f^{-1}(f(p))\subset\C^{n+1}$ is a complex hypersurface,  and $X$ a holomorphic vector field tangent to $V$, one transports as previously the vector field to $\C^n$ and defines the  \emph{relative Poincar\'e-Hopf index} $\Ind_{V,p}(X)$ in the complex case. Note that in the relative complex case, just as in the absolute complex case, the relative index is always positive.

\subsection{Singular points, GSV-index in the complex case} Let as previously, 
$f:(\C^{n+1},0)\to(\C,0)$ be a germ of a holomorphic function. Assume now that $p\in\C^{n+1}$ is an isolated singularity of $f$. Then $V=f^{-1}(0)\subset\C^{n+1}$ is a complex hypersurface with isolated singularity at $p$. Let $X$ be a holomorphic vector field defined in a neighborhood of $p\in \C^{n+1}$
tangent to $V$. That is, relation \eqref{tangent}
holds. In \cite{GSV}, G\'omez-Mont, Seade and Verjovsky defined what is now called the GSV-index of a vector field tangent to a singular variety at the singularity   $\Ind_{V,0}(X)$. 

 In order to formulate the definition, let us first recall that the holomorphic  function $f:(\C^{n+1},0)\to(\C,0)$ having an isolated singularity at the origin defines a Milnor fibration: $f:B\setminus \{0\}\to \C^*$, where $B\subset\C^{n+1}$ is a small ball around the origin. Denote $V_\epsilon=f^{-1}(\epsilon)$. For $\epsilon\ne0$ small, close enough to zero, all fibers $V_\epsilon\cap B$ are isotopic. Note that the vector field 
 $X$ is not necessarily tangent to the fibers $V_\epsilon\cap B$, for $\epsilon\ne0$. We modify X slightly, giving a $C^\infty$ vector field $X_\epsilon$ tangent to a fiber $V_\epsilon\cap B$, for $\epsilon\ne0$ 
close to zero.  We assume moreover that the restriction of the vector field $X_\epsilon$ on $\partial(V_\epsilon\cap B)$ is isotopic to the restriction of the vector field $X$ to $\partial(V\cap B)$ see \cite{P} and \cite{BG}.

The GSV-index can be defined by the formula

\begin{equation}\label{GSV}
\Ind_{V,0}(X)=\sum_{p_i(\epsilon)\in V_{\epsilon}\cap B}\Ind_{V_\epsilon,p_i(\epsilon)}(X_\epsilon).
\end{equation}
 It follows from the Poincar\'e-Hopf theorem that the definition is independent of all choices. Indeed, the 
 Poincar\'e-Hopf theorem says that the right-hand side of \eqref{GSV} is the Euler characteristic 
 $\chi(V_\epsilon\cap B)$ up to some correction term given by the behavior of any vector field 
 $X_\epsilon$ on $\partial(V_\epsilon\cap B)$. Note that by the Milnor fibration theorem all regular fibers $V_\epsilon\cap B$, $\epsilon\ne0$, have the same Euler characteristic. Moreover, the behavior of any vector field in $X_\epsilon$ on $\partial(V_\epsilon\cap B)$
is the same as the behavior of $X$ on $\partial(V\cap B)$. Hence the correction term is independent of the choices.  
 
 For an equivalent topological definition using residues see Suwa \cite{Su}.
 
 \begin{proposition}\label{uptoconstant}\label{char}\cite{BG}
 Up to a constant $K(V)$  independent of the vector field $X$, the GSV-index $\Ind_{V,0}(X)$ is characterized  by the two following conditions:

\begin{description}
\item[(i)] At smooth points $p\in V$, the GSV-index coincides with the relative Poincar\'e-Hopf index $\Ind_{V,p}(X)$. 
\item[(ii)] The GSV-index satisfies the \emph{law of conservation of number}: For any holomorphic vector field $X'$ tangent to 
$V$ sufficiently close to $X$ the following law of conservation of number holds:
\begin{equation}\label{conservation}
\Ind_{V,0}(X)=\sum_{p_{i}\in V} \Ind_{V,p_{i}(\epsilon)}(X').
\end{equation}
Here $p_{i}$ are singularities of $X'$ belonging to $V$, which are close to $0$.
\end{description}
\end{proposition}

The constant can be determined by calculating the GSV-index $\Ind_{V,0}(X)$, for any vector field tangent to $V$.

\subsection{GSV-index in the real case}
Let now $f:(\R^{n+1},0)\to(\R,0)$ be a germ of a real analytic function. 
In this case, there is no Milnor fibration, or more precisely there are two Milnor fibrations:
one for strictly positive small values of $\epsilon$ and one for small strictly negative values of $\epsilon$.
The Euler characteristic of all fibers $V_\epsilon\cap B$, for $\epsilon$ small of the same sign are the same, but can be different for $\epsilon$ positive or $\epsilon$ negative. (Think of $f:\R^3\to\R$ given by
$f(x,y,z)=x^2+y^2-z^2$.)
As in the complex case, in the real case one now defines the GSV-index. More precisely, one defines two GSV indices $\Ind_{V^\pm,0}(X)$, taking $V_\epsilon$, for $\epsilon$ positive or negative respectively. 

\section{Calculating the GSV-index on complex hypersurfaces}
A formula for calculating the GSV-index in the complex case was given by G\'omez-Mont in \cite{G}.
Let us first define the principal ingredients. Let $\OO_{\C^{n+1},0}$ be the algebra of germs of holomorphic functions at the origin. Let $f\in\OO_{\C^{n+1},0}$ be given, with $f(0)=0$. Let $f_i=\frac{\partial f}{\partial z_i}$, $i=0,\ldots,n$, be the partial derivatives of $f$. 
Assume that $0$ is an isolated singularity of $f$. This means that the algebra 
\begin{equation}\label{A}
\AC=\frac{\OO_{\C^{n+1},0}}{(f_0,\ldots,f_n)}
\end{equation}
is finite dimensional. Here $\OO_{\C^{n+1,0}}$ is the algebra of germs at $0$ of holomorphic functions. The dimension $\mu=\dim(\AC)$ is the Milnor number of the singularity. Let $X$ be a germ of holomorphic vector field at $0\in\C^{n+1}$ given by \eqref{X}. Assume that $0$ is an isolated singularity of $X$. This means that the algebra 
\begin{equation}\label{B}
\BC=\frac{\OO_{\C^{n+1},0}}{(X^0,\ldots,X^n)}
\end{equation}
is finite dimensional. Its dimension $\dim(\BC)$ is the Poincar\'e-Hopf index $\Ind_0(X)$ of the vector field $X$ in the ambient space. 

Let $V=f^{-1}(0)$ be the hypersurface defined by $f$ and assume that $X$ is tangent to $V$. That is, 
\eqref{tangent} holds for some holomorphic function $h$.

\begin{theorem}\label{Xavier}\cite{G} The GSV-index of a holomorphic vector field $X$ tangent to a complex hypersurface $V$ at an isolated singularity $0$ is given by.
\begin{equation}
\Ind_{V,0}(X)=
 \begin{cases}
          \dim\frac{\BC}{(f)}-\dim\frac{\AC}{(f)}, & \text{if } $(n+1)$ \text{ even},\\
                  \dim \BC-\dim\frac{\BC}{(h)}+\dim\frac{\AC}{(f)}, & \text{if } $(n+1)$ \text{ odd}.
                   \end{cases}
\end{equation}
\end{theorem}
We give the idea of proof of Theorem \ref{Xavier}. As recalled in Proposition~\ref{uptoconstant}, the GSV index is defined up to a constant by condition (i) and (ii) in Proposition \ref{uptoconstant}.
In \cite{G} G\'omez-Mont 
considers the \emph{Koszul complex}:
\begin{equation}\label{Koszul}
0\to\Omega_{V,0}^{n-1}{\to}\Omega_{V,0}^{n-1}\to\cdots\to\Omega_{V,0}^1\to\OO_{V,0}\to0,
\end{equation}
where 
\begin{equation}\label{omega}
\Omega_{V,0}^i=\frac{\Omega_{\C^{n+1},0}}{f\Omega_{\C^{n+1},0}+df\wedge\Omega_{\C_{n+1,0}^{i-1}}}.
\end{equation}
is the space of \emph{relatively exact forms} on $V$ and the arrows in \eqref{Koszul} are given by contraction of forms by the vector field $X$.
G\'omez-Mont defines the \emph{homological index} $\Ind^{hom}_{V,0}$ as the Euler characteristic of the complex \eqref{Koszul}:
\begin{equation}\label{hom}
\Ind^{hom}_{V,0}=\sum_{i=0}^{n-1}(-1)^i\dim H_i(K)
\end{equation}
where $H_i(K)$, $i=0,\dots,n-1$, are the $i$-th homology groups of the Koszul complex \eqref{Koszul}.
It is easy to see that at smooth points the homological index coincides with the relative Poincar\'e-Hopf index. In \cite{GGK}, Giraldo and G\'omez-Mont show that the homological index verifies the law of conservation (ii) of Proposition \ref{char}.
Hence, the homological index coincides with the GSV-index up to a constant $K(V)$.
The homological index has the advantage that it can be calculated using projective resolutions of a double complex. The horizontal complexes in the double complex are obtained as a \emph{mapping cone} induced by multiplication by the cofactor $h$ in \eqref{tangent} in the Koszul complex in the ambient space. Vertical complexes are obtained as the mapping cone induced by multiplication by $f$  in the \emph{de Rham complex} in the ambient space.  To show that 
the homological index $\Ind^{hom}_{V,0}$ coincides with the GSV-index $\Ind_{V,0}$, 
it is sufficient for each $f$ to calculate both indices on a vector field $X$ associated to $f$. 
If the dimension of the ambient space  $(n+1)$ is even, a natural candidate is the Hamiltonian vector field

\begin{equation}\label{Xf}
X_f=\sum_{i=1}^{(n+1)/2}[f_{2i}\frac{\partial}{\partial x_{2i-1}}-f_{2i-1}\frac{\partial}{\partial x_{2i}}].
\end{equation}
If $(n+1)$ is odd, G\'omez-Mont uses the vector field
\begin{equation}\label{Yf}
Y_f=f\frac{\partial}{\partial x_0}\sum_{i=1}^{(n+1)/2}[f_{2i}\frac{\partial}{\partial x_{2i-1}}-f_{2i-1}\frac{\partial}{\partial x_{2i}}]
\end{equation}
in generic coordinates $x_i$.

\section{Calculating the Poincar\'e-Hopf index of vector fields in $\R^{n+1}$}
When studying the Poincar\'e-Hopf index in the real case, one has to take into account orientation and not just multiplicity. This is done using some bilinear forms. We recall in this section the results of Eisenbud, Levine \cite{EL} and Khimshiashvili \cite{Kh} who solve this problem for real vector fields in the ambient space $\R^{n+1}$. This, in addition to G\'omez-Mont's formula for calculating the GSV-index on complex hypersurfaces, are the two main ingredients in our study. 

Let 
\begin{equation}\label{BR}
\B=\frac{{\mathcal A}_{\R^{n+1},0}}{(X^0,\ldots,X^n)},
\end{equation}
where ${\mathcal A}_{\R^{n+1},0}$ is the algebra of germs at $0$ of analytic functions in $\R^{n+1}$. Let 
$X$, given by \eqref{X}, be a germ of analytic vector field with an algebraically isolated singularity. That is, the singularity when considered over the complex domain remains isolated.  Then the algebra $\B$ is finite dimensional. Let $J=\det(\frac{\partial X^i}{\partial x_j})\in {\mathcal A}_{\R^{n+1},0}$ be the Jacobian of the map defined by the vector field $X$. 
It can be shown that the class $[J]\in \B$ of $J$ in $\B$ 
is non-zero. 
In \cite{EL} and \cite{Kh} Eisenbud, Levine and Khimshiashvili define a nondegenerate bilinear form 
$<\,,\,>_{\B,J}$ as follows. 
\begin{equation}\label{bilin}
\B \times\B{\stackrel {\cdot}\longrightarrow} \B{\stackrel{
L}\longrightarrow } \R.
\end{equation}
Here the first arrow is simply multiplication in the algebra $\B$  and $L$ is any linear mapping such that $L([J])>0$.
Of course, the bilinear form depends on the choice of $L$. However its signature 
$\sgn(\B,J)=\sgn(<\,,\,>_{\B,J})$
does not. More precisely Eisenbud, Levine, Khimshiashvili show

\begin{theorem}\label{ELK} Let $X$ be a germ at $0$ of a real analytic vector field on $\R^{n+1}$ having an algebraically isolated singularity at the origin. Then the Poincar\'e-Hopf index $\Ind_{\R^{n+1},0}(X)$ of 
the vector filed $X$
at the origin is given by 
\begin{equation}\label{eqELK}
\Ind_{\R^{n+1},0}(X)=\sgn(\B,J).
\end{equation}
\end{theorem}

In order to prove the theorem, one has to prove that the signature $\sgn(\B,J)$ coincides with the Poincar\'e-Hopf index for simple singularities and verifies the law of conservation of number. 
The first claim is easily verified. The key-point of the proof of the law of conservation of number is the claim that the bilinear form $<\,,\,>_{\B,J}$ is nondegenerate. 

Once one knows that the form is nondegenerate, the law of conservation of number will follow. Indeed, let $X'$ be a small real deformation of the vector field $X$.  As the bilinear form is nondegenerate, its signature does not change by a small deformation. 
The local algebra $\B$ will decompose into a multilocal algebra $\B(X')$ of the same dimension concentrated in some real point and complex conjugated pairs of points. One verifies that the contribution to the signature of the pairs of complex conjugated points is zero.
From the preservation of signature, there follows the law of conservation of number once one knows that the bilinear form is nondegenerate.

The nondegeneracy of the form $<\,,\,>_{\B,J}$ is a more general feature. It follows from the fact that $J$
generates the \emph{socle} of the algebra $\B$. By definition a socle in an algebra is the \emph{minimal} 
nonzero ideal of the algebra. 

In general, let $\B$ be a real algebra. Assume that the socle of $\B$ is one-dimensional generated by $J\in\B$. We can define a bilinear form $<\,,\,>_{\B,J}$ as above. Following the proof of Eisenbud-Levine in \cite{EL} one verifies that the form $<\,,\,>_{\B,J}$ is nondegenerate. Its signature does not depend on the choice of the linear map $L$ such that $L(J)>0$.

\begin{example} Consider for instance $\B=\frac{{\mathcal A}_{\R^2,0}}{(x^2,y^3)}$. Then the socle is one-dimensional generated by $J=xy^2$.  
The bilinear form $<\,,\,>_{\B,J}$ is a nondegenerate form on the six dimensional space $\B$.

If $\B=\frac{{\mathcal A}_{\R^2,0}}{(x^2,xy^2,y^3)}$, then the socle is generated by $xy$ and $y^2$. It is not one-dimensional and one cannot define a nondegenerate bilinear form as above. 
\end{example}

\section{Bilinear Forms on Local Algebras}
Let $\B=\frac{{\mathcal A}_{\R^{n+1},0}}{(X^0,\ldots,X^n)}$ be a finite dimensional complete intersection algebra. This assures that its socle is one-dimensional  generated by the Jacobian $J=\det(\frac{\partial X^i}{\partial x_j})$. 

In \cite{GM1}, we observed that the Eisenbud-Levine, Khimshiashvili signature generalizes.
Let $h\in\B$ be arbitrary. 
Denote $\Ann(h)=\{g\in\B: gh=0\}$ the \emph{annihilator ideal}
of $h$. 
For $\B$ as above, the algebra $\frac{\B}{\Ann(h)}$ has a one-dimensional socle generated by the element $\frac{J}{h}\in\frac{\B}{\Ann(h)}$. 
The assumption that $(J)$ is minimal guarantees that $J$ can be divided by $h$.
We define the bilinear form $<\,,\,>_{\B,h,J}$ on $\frac{\B}{\Ann(h)}$ by
$<b,b'>_{\B,h,J}=L(bb'h)$, where $L:\B\to\R$ is a linear mapping such that $L(J)>0$. In other words
$<b,b'>_{\B,h,J}=L_h(bb')$, where $L_h(\frac{J}{h})>0$ is a linear mapping. 
Note that in general the element $\frac{J}{h}$ is not well defined in $\B$. However, the ambiguity is lifted in the quotient space $\frac{\B}{\Ann(h)}$.

We put
\begin{equation}\label{rel}
\sgn(\B,h,J)=\sgn<\,,\,>_{\B,h,J}=\sgn(\frac{\B}{\Ann(h)},\frac{J}{h}).
\end{equation}
\subsection{Signatures associated to a singular point of a hypersurface}
Let now $f:(\R^{n+1},0)\to(\R,0)$ be a germ of analytic function having an algebraically isolated singularity at the origin. Let $f_i=\frac{\partial }{\partial x_i}$ be the partial derivatives of $f$.
Consider the  local algebra $\A=\frac{{\mathcal A}_{\R^{n+1},0}}{(f_0,\ldots,f_n)}$. It is a finite complete intersection algebra. Its socle is one-dimensional generated by the Hessian $\Hess(f)=\det(\frac{\partial^2 f}{\partial x_i\partial x_j})$.

Define a flag of ideals in $\A$

\begin{equation}\label{Km}
K_m=\Ann_\A(f)\cap(f^{m-1}), \quad m\geq1.
\end{equation}
Note that
\begin{equation}\label{filtration}
0\subset K_{\ell+1}\subset\cdots\subset K_1\subset K_0=\A.
\end{equation}
Define a family of bilinear forms $<\,,\,>_{f,m}:K_m\times K_m\to\R$ by
\begin{equation}\label{<>}
<a,a'>_{f,m}=<\frac{a}{f^{m-1}},a'>, \quad m=0,\ldots,\ell+1,
\end{equation}
where $<\,,\,>_{\A,Hess(f)}$ is the bilinear form defined in \eqref{bilin} for some linear map $L$ with $L(\Hess(f))>0$. In particular $<a,a'>_{f,0}=<fa,a'>_{\A,\Hess(f)}$. The form $<\,,\,>_{f,0}$ degenerates on $\Ann_\A(f)$, but on $K_0/K_1$ defines a nondegenerate form. We have 
$<a,a'>_{f,1}=<a,a'>_{\A,\Hess(f)}$. This form degenerates on $K_2=\Ann_\A(f)\cap(f)$ etc.
In \cite{GGM2}, we define 

\begin{equation}\label{sigma}
\sigma_i=\sgn<\,,\,>_{f,i}, \quad i=0\ldots,\ell.
\end{equation} 
The signatures $\sigma_i$ are intrinsically associated to the singularity $0$ of $f$. 

\section{Main Result}
The following theorem resumes our results \cite{GM1}, \cite{GM2}, \cite{GGM1}, \cite{GGM2} about the calculation of the GSV-index of singularities of real vector fields on hypersurfaces:

\begin{theorem} \label{main} Let $f:(\R^{n+1},0)\to(\R,0)$ be a germ of analytic function with algebraically isolated singularity at the origin. Let $X$ be an analytic vector field in $\R^{n+1}$ having an algebraically isolated singularity at the origin. Assume that $X$ is tangent to $V=f^{-1}(0)$. That is $X(f)=hf$, for some analytic function $h$. Then
 \begin{description}
\item[(i)] if $(n+1)$ is even,
\begin{equation}\label{even} 
\Ind_{V^+,0}(X)=\Ind_{V^-,0}(X)=\sgn(\B,h(X),J(X))-\sgn(\A,h(X),\Hess(f)).
\end{equation}

\item[(ii)] if $(n+1)$ is odd,
\begin{equation}\label{odd} 
\Ind_{V^\pm,0}(X)=\sgn(\B,h(X),J(X))+\sgn(\A,\Hess(f))+K_\pm, \end{equation}
where 
\begin{equation}
K_+=\sum_{i\geq 1}\sigma_i, \quad K_-=\sum_{i\geq 1}(-1)^i\sigma_i.
\end{equation}

\end{description}
\end{theorem} 

\section{Proof of the Main theorem}

We give here the main ingredients of the proof of Theorem \ref{main}. The GSV-index is determined by three properties:
\begin{description}
\item[(i)] Value at smooth points
\item[(ii)] The law of conservation of number
\item[(iii)] Constants $K_\pm$ depending only on the orientation (side) $V_\pm$ of the variety $V=f^{-1}(0)$ and not on the vector field. 
\end{description}

One verifies easily that at smooth points of $V$,  the formula is valid. Indeed, from the tangency condition there follows $(f)\subset \Ann(h)$. In smooth points the converse is also true. Hence
$\Ann(h)=(f)$. Next, working in a local chart at smooth points one shows that 
$\sgn<\,,\,>_{\B,h,J}$ gives the relative Poincar\'e-Hopf index of the vector field. Then, one has to show that our formulas  \eqref{even}  and \eqref{odd} verify the law of conservation of number.
Some parts are easier in the even case and some other are easier in the odd case.

\subsection{$(n+1)$ odd case}

 The law of conservation of number is easy for $(n+1)$ odd.
Indeed, in this case the complex index, up to a constant depending only of $f$, is  $\dim\B^\C-\dim\frac{\B^\C}{(h)}=\dim\frac{\B^\C}{\Ann(h)}$ (see Theorem \ref{Xavier}). On the other hand on $\frac{\B}{\Ann(h)}$ there is the non-degenerate form $<\,,\,>_{\B,h,J}$. Make a small deformation $X'$ of $X$, tangent to $V$.  The corresponding local algebra $\B$ or rather its complexification decomposes into a multilocal algebra concentrated in several points corresponding to singular points of $X'$. 
The dimension of the multilocal algebra is equal to the sum of the dimensions at points in which it is concentrated. On the other hand, by Theorem \ref{Xavier} of Gomez-Mont, the dimension 
$\dim\frac{\B^\C(X')}{\Ann(h)}$ verifies the law of conservation of number. Hence, the dimension of the 
multilocal algebra obtained after deformation $X'$ of $X$ is 
equal to the dimension of the local algebra $\dim\frac{\B^\C}{\Ann(h)}$ before the deformation. 
This permits to extend continuously the bilinear form $<\,,\,>_{h,J}$ from the algebra $\frac{\B}{\Ann(h)}$
to its deformation. By nondegeneracy of the form $<\,,\,>_{h,J}$, its signature is unchanged by a small deformation. This gives the law of conservation of number for the signature of $<\,,\,>_{h,J}$ when adding the signatures for all (real or complex) singular points of $X'$ appearing after deformation.  Note that from the tangency condition \eqref{tangent}, it follows that $(f)\subset \Ann(h)$, so only points in $V=f^{-1}(0)$ can contribute to the signature $\sgn<\,,\,>_{\B(X'),h,J}$ after deformation. 
At the end, let us note that complex zeros of $X'$ come in pairs. One verifies that the contribution to the signature of each pair is equal to zero. Hence only real singular points of $X'$ belonging to $V$ contribute. 
The law of conservation of number (in the real case) for the formula $\sgn(\B,h,J)$ follows. 

The final step in proving the formula in the case $(n+1)$ odd is to adjust the constant $\sgn(\A,\Hess(f))+K_\pm$. This is difficult in the odd case. We will come back to it in subsection \ref{7.4}.

\subsection{$(n+1)$ even case}
In the $(n+1)$ even case Theorem \ref{Xavier} says that in the complex case, up to a constant, the index is given by $\dim\frac{\B^\C}{(f)}$.  There is no natural bilinear form on  $\frac{\B^\C}{(f)}$.
We consider a non-degenerate bilinear form on $\frac{\B^\C}{\Ann(h)}$. We stratify the space of bilinear vector fields by the dimension of the ideal $(h)$ in the algebra $\A$.  The signature $\sgn(h(X),J(X))$ verifies the law of conservation of number in restriction to each stratum. We show that when changing the stratum the jump in $\sgn(\B,h,J)$ is equal to the jump in $\sgn(\A,h,\Hess(f)).$ The two jumps hence compensate in the index formula \eqref{even}. In order to show the equality of the jumps it is sufficient to study the place where all strata meet i.e. the stratum of highest codimension. One has the highest codimension for the Hamiltonian vector field $X_f$ given by \eqref{Xf}, when $h=0$. Note that in this case the two algebras $\A$ and $\B$ coincide and $J(X)=\Hess(f)$.   

In this case it is very easy to determine the constant (independent of the vector field) adjusting the  signature formula with index. For that purpose, one studies the Hamiltonian vector field $X_f$. Note that the Hamiltonian vector field $X_f$ is tangent to all fibers $V_\epsilon=f^{-1}(\epsilon)$. Moreover, it has the same behavior on the boundary $V_\epsilon\cap B$, for $\epsilon\ne0$  as on $V\cap B$.
The Hamiltonian vector field $X_f$ has no zeros on $V_\epsilon=f^{-1}(\epsilon)$, for $\epsilon\ne0$.
Hence $\Ind_{V_\pm}(X_f)=0$.
On the other hand $\sgn(\B,h(X),J(X))-\sgn(\A,h,Hess)=0$, as $\A=\B$ and $J=\Hess$. It follows that no correction term has to be added to $\sgn(\B,h(X),J(X))-\sgn(\A,h,Hess)$ in order to obtain the formula for $\Ind_{V_\pm,0}(X)$.

\subsection{Why is $\Ind_{V_+,0}(X)=\Ind_{V_-,0}(X)$ in the $(n+1)$ even case and not in the odd case?}

We explain here why $\Ind_{V_+,0}(X)=\Ind_{V_-,0}(X)$ in the $(n+1)$ even case and not in the odd case.
Note first that the index of a vector field in the ambient space is an even function if the dimension $(n+1)$ of the ambient space is even and is an odd function if $(n+1)$ is odd. We next use Morse theory.
Consider the vector field $\nabla f.$ By Morse theory, the Euler characteristic $\chi$ verifes:
\begin{equation}\label{chi}
\left.
\begin{array}{rcl}
  \chi(V_+)& =  &   1+\Ind(\nabla f)\\
 \chi(V_-) &=   & 1+\Ind(-\nabla f).
\end{array}
\right.\end{equation}
Here $\chi(V_{+})$ is the Euler characteristic of $V_\epsilon\cap B$, for $\epsilon>0$ small. The value $\chi(V_{-})$ is defined analogously. 

If $(n+1)$ is even, then $\Ind(\nabla f)=\Ind(-\nabla f)$, so $\chi(V_+)=\chi(V_-)$ and $\Ind_{V_+,0}(X)=-\Ind_{V_-,0}(X)$.

If $(n+1)$ is odd, then $\Ind(-\nabla f)=-\Ind(\nabla f)$, so $\chi(V^+)-1=-(\chi(V_-)-1)$ and 
$\Ind_{V_-,0}(X)=2-\Ind_{V_+,0}(X)$.
\subsection{Adjusting the constant $K$ in the $(n+1)$ odd case}\label{7.4}
In order to complete the sketch of proof of the main theorem, we have to explain how do we calculate the constant $K_\pm$ appearing in the $(n+1)$ odd case \eqref{odd}.

As shown previously, the two signature terms in \eqref{odd} calculate the GSV-index up to a constant independent of the vector field. In order to determine the constant, for each $V=f^{-1}(0)$, one has to take a vector field tangent to $V$, having an algebraically isolated singularity at the origin. 
Contrary to the situation in the $(n+1)$ even case, in the odd case, there is no such natural vector field. As in \cite{G}, we use the family of vector fields
\begin{equation}\label{Xft}
X_t=(f-t)\frac{\partial}{\partial x_0}\sum_{i=1}^{(n+1)/2}[f_{2i}\frac{\partial}{\partial x_{2i-1}}-f_{2i-1}\frac{\partial}{\partial x_{2i}}]
\end{equation}
 in generic coordinates. 
 The local algebra is $\B=\B(X_0)=\frac{\mathcal{A}_{\R^{n+1},0}}{(f,f_1,f_2,\ldots,f_n)}.$
Note that $X_t$ is tangent to $V_t=f^{-1}(t)$, for any $t$. 
More precisely, $X_t(f)=f_0 f$, so $h=f_0$ is the cofactor of $X_t$.
Hence, by definition 
\begin{equation}
\Ind_{V_+}(X,0)=\sum_{p_t\in V_t\cap B}\Ind_{V_t,p_t}(X_t).
\end{equation}
But, these indices are calculated using the multilocal algebra $\B_t$ and the relative Jacobian $\frac{J(X_t)}{f_0}$. That is, the index is given by the signature of the bilinear form $<\,,\,>_{\B_t}$, for $t\ne0$ small.
For the index $\Ind_{V_+,0}(X_0)$, we have to take it positive and for $\Ind_{V_-,0}(X_0)$ it is negative. 
 The problem is that this form degenerates on $\Ann_{\B_t}(f_0)$, for $t=0$.

We prove in \cite{GGM2} a general result for algebras $\A=\A(f)$ and $\B=\B(X)$ associated to  a vector field $X$ tangent to $V=f^{-1}(0)$ i.e. verifying \eqref{tangent}:

\begin{lemma}\label{iso}
There exists a natural isomorphism between the algebras $\Ann_{\B}(h)$ and $\Ann_\A(f)$.
\end{lemma}
\begin{proof}
The isomorphism is given by the mapping $g\mapsto k$ if $gh=fk$.
\end{proof}

Lemma \ref{iso} permits to transport all higher order signature vanishing in $\Ann_{\B_0}$ to a natural algebra $\A$. We apply it to our vector field $X_t$. When looking at the signature of the form $<\,,\,>_{\B_t}$, we have one part which does not degenerate. It is the part in $\frac{\B}{\Ann_\B(f_0)}$. The bilinear form $<\,,\,>_{\B_t}$ degenerates at different orders on different parts of
$\Ann_{\B_t}(f_0)$.  By Lemma \ref{iso}, we transport the bilinear form $<\,,\,>_{\B_t}$ to a bilinear form in the coordinate independent algebra  $\Ann_\A(f)$. Note that in $\B_t$, we have $f=t$, so degeneration of $<\,,\,>_{\B_t}$ at different orders of $t$ corresponds to multiplication by $f$ in the algebra $\Ann_{\A}(f)$. 
For more details see \cite{GGM2}.

\section{Open problems}

\subsection{Geometric interpretation of the signatures $\sigma_i$. Filtration of contributions to the Euler characteristic of the generic fiber}
In Theorem \ref{main} appear higher order signatures $\sigma_i$ defined in \eqref{sigma}.  These signatures are associated to the singularity $f$ alone. We would like to give a geometric interpretation of these numbers. We believe that they correspond to parts of the Euler characteristic of the generic fiber, filtered by the speed of arrival at the singular fiber.

Let us be more precise. In \cite{T} Teissier studies polar varieties in the complex case.
He considers a germ of a function $f:(\C^2,0)\to(\C,0)$ having an isolated critical point at the origin.
He considers a Morsification $f_s=f-s x_0$ of $f$ in generic coordinates $(x_0,\ldots,x_{n})$. Its critical points are given by 

\begin{equation}
f_0-s=f_1=\cdots=f_n=0.
\end{equation} 
Let $\Gamma$ be the curve given by $f_1=\cdots=f_n=0$. The curve $\Gamma$ is called \emph{polar curve}. In general it has several branches $\Gamma=\cup_{q=1}^\ell\Gamma_q$. By Morsification, the critical point $0$ of $f$ decomposes in several critical points arriving along the polar curve to the origin. For each value of $s\ne0$,  the critical points of $f_s$ belong to $f_0^{-1}(s)\cap\Gamma$. 
Each critical point corresponds to a vanishing cycle contributing to $H_n(V_{t_0})$.
In \cite{T}, Teissier observed that, after Morsification, critical points arrive at different speed at the origin. 
More precisely, each component $\Gamma_q$ of the polar curve $\Gamma$ is parametrized as
\begin{equation}\label{Gamma}
\left.
\begin{array}{rcl}
  x_0(t_q)& =& t_q^{m_q}+\cdots \\
  \ldots&\ldots&\ldots\\
 x_n(t_q) &= &\lambda_n t_q^{k_{q,n}} +\cdots
\end{array}
\right.\end{equation}
where $m_q\leq k_{q,i}$. 
In \cite{T}, Teissier calculates the exponent $m_q$. 
One can use $x_0$ (or the corresponding critical value) as a measure for the speed of approach of a critical point in the Morsification. 
One can filter the $n$-th group of homology of the generic fiber $H_n(f^{-1}(t))$ i.e. the space of vanishing cycles, by  the speed of arrival of the corresponding critical points. 
We believe that this filtration is given by  the filtration \eqref{filtration} or rather its complex counterpart.
The dimensions 
\begin{equation}
0=\dim\frac{\A}{K_0}\leq\dim\frac{\A}{K_1}\leq\dim\frac{\A}{K_{\ell+1}}\leq\dim\A
\end{equation}
 would measure the dimension of the space of vanishing cycles arriving at a certain minimal speed. 
 
 The signatures $\sigma_i$ would be the real counterpart. The signature $\sigma_0$ is a signature  of a bilinear form on $\A$. It measures the Euler characteristic $\chi(V_t)$. We believe that the signatures $\sigma_i$ that we introduced measure the filtered part of the Euler characteristic of the generic fiber $\chi(V_t)$, the filtration being done  by taking only the part of the topology of the fiber arriving at a certain minimal speed.
 We hope to be able to address this problem in a continuation of our research.
 
 \subsection{Generalization to higher codimension}
 
 In \cite{BEG} Bothmer, Ebeling and  G\'omez-Mont generalized G\'omez-Mont's formula (Theorem \ref{Xavier}) to a  formula for the index of a vector field on an isolated complete intersection singularity in the complex case. 
 A natural problem would be to extend the result to the real case. Here, as in our Theorem \ref{main}, one would certainly have to define some bilinear forms on the spaces studied in \cite{BEG}.

\bibliographystyle{plain}

\begin{thebibliography}{}

\end{thebibliography}


\begin{thebibliography}{99}

\bibitem{BG} Bonatti, Ch.; G\'omez-Mont, X. The index of holomorphic vector fields on singular varieties. I. Complex analytic methods in dynamical systems (Rio de Janeiro, 1992). Ast\'erisque No. 222 (1994), 3, 9--35. 
\bibitem{BEG} Graf von Bothmer, H.-Ch.; Ebeling, Wolfgang; G\'omez-Mont, Xavier An algebraic formula for the index of a vector field on an isolated complete intersection singularity. Ann. Inst. Fourier  58 (2008), no. 5, 1761--1783.
\bibitem{BSS} Brasselet, Jean-Paul; Seade, Jos\'e; Suwa, Tatsuo Vector fields on singular varieties. Lecture Notes in Mathematics, 1987. Springer-Verlag, Berlin, 2009. xx+225 pp.
\bibitem{EL} Eisenbud, David; Levine, Harold I.
An algebraic formula for the degree of a $C^\infty$ map germ. 
Ann. of Math. (2) 106 (1977), no. 1, 19--44.
\bibitem{E}  Eisenbud, David Commutative algebra. With a view toward algebraic geometry. Graduate Texts in Mathematics, 150. Springer-Verlag, New York, 1995. xvi+785 pp.
\bibitem{G} G\'omez-Mont, Xavier An algebraic formula for the index of a vector field on a hypersurface with an isolated singularity. J. Algebraic Geom. 7 (1998), no. 4, 731--752.

\bibitem{GGK} Giraldo, L.; G\'omez-Mont, X. A law of conservation of number for local Euler characteristic. Complex manifolds and hyperbolic geometry (Guanajuato, 2001), 251--259, Contemp. Math., 311, Amer. Math. Soc., Providence, RI, 2002.
\bibitem{GGM1}
Giraldo, Luis; G\'omez-Mont, Xavier; Marde\v si\'c, Pavao Computation of topological numbers via linear algebra: hypersurfaces, vector fields and vector fields on hypersurfaces. Complex geometry of groups (Olmu\'e, 1998), 175--182, Contemp. Math., 240, Amer. Math. Soc., Providence, RI, 1999.Giraldo, L.; 
\bibitem{GGM3}
Giraldo, L.; G\'omez-Mont, X.; Marde\v si\'c, P. On the index of vector fields tangent to hypersurfaces with non-isolated singularities. J. London Math. Soc. (2) 65 (2002), no. 2, 418--438.
\bibitem{GM1} 
G\'omez-Mont, X.; Marde\v si\'c, P. The index of a vector field tangent to a hypersurface and the signature of the relative Jacobian determinant. Ann. Inst. Fourier (Grenoble) 47 (1997), no. 5, 1523--1539. 
\bibitem{GM2}
Gomes-Mont, Kh.; Mardesich, P. The index of a vector field tangent to an odd-dimensional hypersurface, and the signature of the relative Hessian. (Russian) Funktsional. Anal. i Prilozhen. 33 (1999), no. 1, 1--13, 96; translation in Funct. Anal. Appl. 33 (1999), no. 1, 1--10
\bibitem{GGM2}
Giraldo, L; G\'omez-Mont, X.; Marde\v si\'c, P. Flags in zero dimensional complete intersection algebras and indices of real vector fields. Math. Z. 260 (2008), no. 1, 77--91.
\bibitem{GSV} 
G\'omez-Mont, X.; Seade, J.; Verjovsky, A. The index of a holomorphic flow with an isolated singularity. Math. Ann. 291 (1991), no. 4, 737--751.
\bibitem{Kh}  Him\v sia\v svili, G. N. The local degree of a smooth mapping. (Russian) Sakharth. SSR Mecn. Akad. Moambe 85 (1977), no. 2, 309--312.
\bibitem{P} Pugh, Charles C. A generalized Poincar\'e index formula. Topology 7 1968 217--226. 
\bibitem{S}  Seade, Jos\'e On the topology of isolated singularities in analytic spaces. Progress in Mathematics, 241. Birkh\"auser Verlag, Basel, 2006. xiv+238 pp.
\bibitem{Su} Suwa T. GSV-indices as residues, 1--12 this volume 
\bibitem{T} Teissier, B. Vari\'et\'es polaires. I. Invariants polaires des singularit\'es d'hypersurfaces.  Invent. Math. 40 (1977), no. 3, 267--292.
\end{thebibliography}

\end{document}